\documentclass[a4paper,DIV=classic]{myart}
\usepackage[normalem]{ulem}

\newcommand{\norm}[1]{\left\Vert#1\right\Vert}
\newcommand{\abs}[1]{\left\vert#1\right\vert}

\newcommand{\set}[1]{\ensuremath{ \{ #1 \} }}
\newcommand{\R}{\mathbb{R}}
\newcommand{\N}{\mathbb{N}}

\newcommand*{\cadlag}{c\`adl\`ag}

\newcommand{\brak}[1]{\left(#1\right)}    
\newcommand{\crl}[1]{\left\{#1\right\}}   
\newcommand{\edg}[1]{\left[#1\right]}     

\DeclareMathOperator*{\argmin}{arg\,min}

\begin{document}
\title{Minimal Supersolutions of Convex BSDEs under Constraints}
\author[a,1,s]{Gregor Heyne}
\author[b,2]{Michael Kupper}
\author[a,3]{Christoph Mainberger}
\author[c,4]{Ludovic Tangpi}

\address[a]{Humboldt-Universit\"at zu Berlin, Unter den Linden 6, 10099 Berlin, Germany}
\address[b]{University of Konstanz, Universit\"atsstr.~10, 78457 Konstanz, Germany}
\address[c]{University of Vienna, Faculty of Mathematics, Oskar-Morgenstern-Platz 1, A-1090}

\eMail[1]{heyne@math.hu-berlin.de}
\eMail[2]{kupper@uni-konstanz.de}
\eMail[3]{mainberg@math.hu-berlin.de}
\eMail[4]{ludovic.tangpi@univie.ac.at}

\myThanks[s]{Financial support: MATHEON project E.2}

\abstract{
We study supersolutions of a backward stochastic differential equation, the control processes of which 
are constrained to be continuous semimartingales of the form  $dZ= \Delta dt + \varGamma dW$.%
~The generator may depend on the decomposition $(\Delta,\varGamma)$ and is assumed to be positive, jointly
convex and lower semicontinuous, and to satisfy a superquadratic growth condition in $\Delta$ and $\varGamma$.%
~We prove the existence of a supersolution that is minimal at time zero
and derive stability properties of the non-linear operator that maps terminal
conditions to the time zero value of this minimal supersolution such as monotone convergence, Fatou's lemma
and $L^1$-lower semicontinuity.
~Furthermore, we provide duality results within the present framework 
and thereby give conditions 
for the existence of solutions under constraints.
}
\keyWords{Supersolutions of Backward Stochastic Differential Equations; Gamma Constraints; Minimality under Constraints; Duality}
\keyAMSClassification{60H20; 60H30}

\maketitle

\section{Introduction}
\addcontentsline{toc}{section}{Introduction}
\markboth{\uppercase{Introduction}}{\uppercase{Introduction}}

On a filtered probability space, the filtration of which is generated by a $d$-dimensional Brownian
motion, we are interested in quadruplets $(Y,Z,\Delta,\varGamma)$ of processes
such that, for all $0\le s\le t\le T$, the system
\begin{align}\label{eq_intro}
 Y_s& - \int_s^t g_u(Y_u,Z_u,\Delta_u,\varGamma_u)du + \int_s^t Z_u dW_u \ge Y_t\,, \qquad Y_T \ge \xi\,, \nonumber\\
 Z_t& = z + \int_0^t\Delta_udu + \int_0^t\varGamma_udW_u 
\end{align}
is satisfied.
Here, for $\xi$ a terminal condition, $Y$ is the \cadlag\, value process 
and $Z$ the continuous control process with decomposition $(\Delta,\varGamma)$.  
The generator $g$ is assumed to be jointly convex and may depend on the decomposition of the continuous semimartingale $Z$.
It is our objective to give conditions ensuring that the set $\mathcal A(\xi,g,z)$, consisting of 
all admissible pairs $(Y,Z)$ satisfying \eqref{eq_intro}, henceforth called supersolution of the backward stochastic
differential equation (BSDE) under gamma and delta constraints, 
contains elements $(\hat Y,\hat Z)$ that are minimal at time zero. 
Furthermore, we give conditions relying on BSDE duality for the existence of solutions under constraints.\\

Finding the minimal initial value of a supersolution under constraints 
is closely related to
the superreplication problem in a financial market under gamma constraints, first studied in \citet{SonTouz02}.
Indeed, the classical gamma constraints can be incorporated into our more general framework by
setting the generator to $+\infty$ whenever the diffusion part $\varGamma$ 
is outside a predetermined interval.
In \citet{SonTouz02}, the decomposition parts of the trading strategies are assumed to be bounded and under similar assumptions, \citet{CherSonTouz} focus on the multidimensional case. 
In both papers the problem is formulated in a dynamic setting, allowing the authors to use dynamic programming tools to identify the value functions as unique viscosity solutions of parabolic partial differential equations.
The present paper in contrast focuses on the static case and studies the problem using purely probabilistic tools.
Moreover, instead of a priori bounding the components of the control process, we opted for incorporating a growth condition on the generator which in turn ensures that our controls belong to suitable spaces.
Let us briefly outline the idea behind our approach.
In a nutshell, inspired by the methods first used in \citet{CSTP} and then later in \citet{HKM}, we begin by
considering the operator $\mathcal E_0^g(\xi,z):=\inf \{Y_0\,:\, (Y,Z) \in \mathcal A(\xi,g,z)\}$ 
where $z\in\R^{1\times d}$ is the initial value of controls. 
We then show that 
the set of supersolutions $(\hat Y,\hat Z)$ satisfying $\hat Y_0=\mathcal E_0^g(\xi,z)$ is non-empty.
In order to do so, we impose  a superquadratic growth condition in the decomposition parts $(\Delta,\varGamma)$ of controls on the generator $g$, 
reflecting a penalization of rapid changes in control values and accounting for the expression ``Delta- and Gamma-Constraints''.
The consequence is twofold.
First, it ensures that the sequence of stochastic integrals $(\int Z^ndW)$ corresponding to the minimizing sequence 
$Y^n_0\downarrow\mathcal E_0^g(\xi,z)$ is bounded in $\mathcal H^2$.
Drawing from compactness results for the space of martingales $\mathcal H^2$ given in \citet{DelbSchach01}, 
we obtain our candidate control process $\hat Z$
as the limit of a sequence in the asymptotic convex hull of $(Z^n)$.
At this point it is crucial to preserve the continuous semimartingale structure of the limit object,
that is $\hat Z = z + \int \hat\Delta du + \int \hat\varGamma dW$.
Here the first novelty of this paper comes into play since, although following the ideas used in \citet{CSTP},
it is a priori not clear that the candidate control has the right structure.
To achieve this, we prove two auxiliary results by using once more the aforementioned growth condition on $g$.

In a next step, we provide stability results of $\xi\mapsto\mathcal E^g_0(\xi,z)$, the non-linear operator that maps a terminal
condition to the value of the minimal supersolution at time zero, such as monotone convergence,
Fatou's lemma or $L^1$-lower semicontinuity.
This, together with convexity, gives way to a dual representation of $\mathcal E^g_0$ as a consequence of the Fenchel-Moreau theorem,
which is the second main novel contribution of this work.
Indeed, we use purely probabilistic methods in order to characterize the conjugate $\mathcal E^*_0$ in terms of the decomposition parts of the controls and show that $\mathcal E^*_0$ is 
always attained.
Note that, in contrast to \citet{DrapeauTangpi}, in the presence of constraints identifying the convex conjugate $\mathcal E^*_0$ is technically more involved.
In particular, for the case of a quadratic generator we show that it is possible to explicitly compute the conjugate
by means of classical calculus of variations methods, giving additional structural insight into the problem.
If we assume in turn the existence of an optimal subgradient such that $\mathcal E^g_0(\xi,z)$ is attained in its dual representation, we can
prove that the associated BSDE with parameters $(\xi,g)$ admits a solution under constraints.
Our duality results extend those of \citet{DelbaenHu} and \citet{DrapeauTangpi} obtained in the unconstrained case
as the existence of constraints require new methods in order to characterize the convex conjugate.\\

Before we continue, let us briefly discuss the existing literature on the subject.
Ever since the seminal paper \citet{peng01}, an extensive amount of work has been done in the field of BSDEs,
resulting in such important contributions as for instance \citet{karoui01}, \citet{kobylanski01} or \citet{briand03}.
We refer the reader to \citet{peng99} or \citet{CSTP} for a more thorough treatment of the literature concerning
solutions and in particular supersolutions of BSDEs.
There are many works dealing with optimization or (super-)replication under constraints,
see for instance \citet{CvitanicKarat01}, \citet{JouiniKallal01} or \citet{BroadieCvitanic01} and references therein,
but the notion of gamma constraints in the context of superhedging was introduced in \citet{SonTouz02} and then studied in a multi-dimensional setting in \citet{CherSonTouz}.
We would also like to refer the reader to \citet{CherSonTouz02}, where the authors treat the related system of
BSDEs and SDEs in a more abstract fashion, whereas the more recent work \citet{SonTouzZhang} provides a dual characterization of the superreplication problem.\\

The remainder of this paper is organized as follows. 
Setting and notations are specified in Section \ref{sec01}. 
A precise definition of supersolutions under gamma and delta constraints is then given 
in Section \ref{sec02}, along with existence and stability results. 
We conclude this work with duality results 
in Section \ref{sec023}.

\section{Setting and notations}\label{sec01}
We consider a filtered probability space
$(\Omega,\mathcal F, (\mathcal F_t)_{t\ge 0},P)$,
where the filtration $(\mathcal F_t)$ is generated by a
$d$-dimensional Brownian motion $W$ and is assumed to satisfy the usual conditions.%
~For some fixed time horizon $T>0$ 
the set of $\mathcal F_T$-measurable random
variables is denoted by $L^0$, where random variables
are identified in the $P$-almost sure sense.
By $L^p$ we furthermore denote the set of random variables in $L^0$ 
with finite $p$-norm,
for $p \in [ 1,+\infty]$.
Inequalities and strict inequalities between any two
random variables or processes $X^1$ and $X^2$ are understood in the $P$-almost sure or
in the $P\otimes dt$-almost everywhere sense, respectively.
We denote by $\mathcal{T}$ the set of stopping times with values in $[0,T]$ 
and hereby call an increasing sequence of stopping times $(\tau^n)$ such that
$P[\bigcup_{n}\set{\tau^n=T}]=1$ a localizing sequence of stopping
times.
For $m,n\in\N$ we denote by $|\cdot|$ the Euclidean norm on $\R^{m\times n}$, that is $|x|=(\sum_{i,j}x^2_{ij})^{\frac{1}{2}}$.
By $\mathcal{S}:=\mathcal{S}(\R)$ we denote the set of 
\cadlag\, progressively measurable processes $Y$ with values in $\R$.
For $p \in \left[ 1,+\infty \right[$, we further denote by
$\mathcal{H}^p$ the set of \cadlag\, local martingales $M$ with finite $\mathcal H^p$-norm on $[0,T]$, that is 
$\norm{M}_{\mathcal H^p} := E[\langle M,M \rangle_T^{p/2}]^{1/p} <\infty$.
By $\mathcal{L}^p:=\mathcal{L}^p\left( W \right)$ we denote the set of $\R^{1\times d}$-valued, progressively
measurable processes $Z$ such that $\int ZdW \in \mathcal H^p$, that is,
$\norm{Z}_{\mathcal{L}^p}:=E[(\int_{0}^{T}\abs{Z_s}^2 ds)^{p/2}]^{1/p}$ is finite.
For $Z \in \mathcal{L}^p$, the stochastic integral $\int Z dW$ is well defined, see \citep{protter}, and is by means
of the Burkholder-Davis-Gundy inequality \citep[Theorem 48]{protter} 
a continuous martingale.
We further denote by $\mathcal{L}:=\mathcal{L}\left( W \right)$ the set of
$\R^{1\times d}$-valued, progressively measurable processes $Z$ such that
there exists a localizing sequence of stopping times $(\tau^n)$ with
$Z1_{\left[0,\tau^n  \right]} \in \mathcal{L}^1$, for all $n\in\N$.
For $Z\in \mathcal L$, the stochastic integral $\int_{}^{}ZdW$ is well defined and is a
continuous local martingale.
Finally, for a given sequence $(x_n)$ in some convex set, we say that a sequence $(\tilde x_n)$ is in the asymptotic
convex hull of $(x_n)$ if $\tilde x_n \in conv\{x_n,x_{n+1},\dots\}$, for all $n\in\N$.\\

\section{Minimal supersolutions of BSDEs under delta and gamma constraints}\label{sec02}
\subsection{Definitions}\label{sec21}
Throughout this work, a generator 
is a jointly measurable function $g$ from $\Omega\times[0,T]\times \R\times
\R^{1\times d} \times \R^{1\times d} \times \R^{d\times d}$ to $\R\cup\{+\infty\}$
where $\Omega\times[0,T]$ is endowed with the progressive $\sigma$-field.
A control $Z\in\mathcal L$ with initial value $z\in\R^{1\times d}$ is said to have the 
decomposition $(\Delta,\varGamma)$ if it is of the form $Z=z+\int \Delta du + \int \varGamma dW$, for 
progressively measurable $(\Delta,\varGamma)$ taking values in $\R^{1\times d}\times\R^{d\times d}$.\footnote{
In order to be compatible with the dimension of $Z$, actually the transpose $(\int \varGamma dW)^T$
of $\int\varGamma dW$ needs to be considered.
However, we suppress this operation for the remainder in order to keep the notation simple.}
A control is said to be admissible if the continuous local
martingale $\int Z dW$ is a supermartingale.
Let us collect all these processes in the set $\Theta$ defined by
\begin{equation*}
\Theta:= \left\{ Z\in \mathcal L\,: \begin{array}{l}   \mbox{there exists $z\in\R^{1\times d}$ and progressively measurable $(\Delta,\varGamma)$  such that}
\enspace \\ \mbox{$Z = z + \int\Delta du + \int\varGamma dW$ and $\int ZdW$ is a supermartingale}
\end{array} \right\} \,.
\end{equation*}
Whenever we want to stress the dependence of controls on a fixed initial value $z\in\R^{1\times d}$, we make use of 
the set $\Theta(z):=\{Z\in\Theta:Z_0=z\}$. 
Given a generator $g$ and a
terminal condition $\xi\in L^0$, a pair $(Y,Z)\in{\cal S}
\times\Theta$ is a supersolution of a BSDE under gamma and delta constraints if, for $0\leq s\leq
t\leq T$, it holds
\begin{equation}\label{eq03}
Y_s - \int_s^t g_u(Y_u,Z_u,\Delta_u,\varGamma_u)du + \int_s^t Z_u dW_u \ge Y_t \quad\text{and}\quad Y_T \ge
\xi\,\,.
\end{equation}
For a supersolution $(Y,Z)$, we call $Y$ the value process and $Z$ its
corresponding control process.\footnote{Note that the formulation in \eqref{eq03} is equivalent to the existence of a
\cadlag\,increasing process $K$, with $K_{0}=0$, such that
$	Y_{t}=\xi + \int_t^T g_u(Y_u,Z_u,\Delta_u,\varGamma_u)du+(K_{T}-K_{t})-\int_t^T
Z_{u}dW_{u}$ for all $t\in[0,T]$, see for example \cite{karoui01, peng99}.}
Given $z\in\R^{1\times d}$, we are now interested in the set
\begin{equation*}
\mathcal A(\xi,g,z) := \crl{(Y,Z) \in\mathcal S\times \Theta(z)\,: \,\mbox{\eqref{eq03} holds}}\,.
\end{equation*}
Throughout this paper a generator $g$ is said to be 
\begin{enumerate}[label=\textsc{(lsc)},leftmargin=40pt]
 \item  if  $(y,z,\delta,\gamma)\mapsto g(y,z,\delta,\gamma)$ is lower semicontinuous.\label{lsc} 
\end{enumerate}
\begin{enumerate}[label=\textsc{(pos)},leftmargin=40pt]
 \item positive, if  $g(y,z,\delta,\gamma)\ge 0$, for all $(y,z,\delta,\gamma) \in \R\times\R^{1\times d}\times\R^{1\times d}\times \R^{d\times d}$.\label{pos}
\end{enumerate}
\begin{enumerate}[label=\textsc{(con)},leftmargin=40pt]
 \item convex, if $(y,z,\delta, \gamma)\mapsto g(y,z,\delta,\gamma)$ is jointly convex.\label{con}
 \end{enumerate}
\begin{enumerate}[label=\textsc{(dgc)},leftmargin=40pt]
 \item delta- and gamma-compatible, if there exist $c_1\in\R$ and $c_2>0$ such that, for all $(\delta,\gamma)\in\R^{1\times d}\times \R^{d\times d}$,
\[g(y,z,\delta,\gamma) \ge c_1 + c_2\brak{|\delta|^2 + |\gamma|^2}\]
holds for all $(y,z)\in \R\times\R^{1\times d}$.\label{dgc}
\end{enumerate}
\begin{remark}\text{}\
\begin{enumerate}
\item[(i)] Note that \ref{dgc} reflects a penalization 
of rapid changes in control values. 
In contrast to \cite{CherSonTouz} or \cite{CherSonTouz02}, where the single decomposition parts $\Delta$ and $\varGamma$ were 
demanded to satisfy certain boundedness, continuity or growth properties, we embed this 
in $\ref{dgc}$ so that suitable $\mathcal L^2$-bounds emerge naturally from the problem \eqref{eq03}.
\item[(ii)] An example of a generator that excludes values of $\varGamma$
exceeding a certain level by penalization and fits into our setting is given by
\[
g(y,z,\delta,\gamma) = \left\{ \begin{array}{ll}
            \tilde g(y,z,\delta)\enspace\,\,&\text{if}\enspace |\gamma|\le M\\ 
            +\infty\enspace&\text{else}
           \end{array} \right.\,,
\]
where $M>0$ and $\tilde g$ 
is any positive, jointly convex and lower semicontinuous generator satisfying 
$\tilde g(y,z,\delta) \ge c_1 +c_2 |\delta|^2 $ for constants $c_1\in\R$ and $c_2>0$.  
This particular choice of $g$ is closely related to the kind of gamma constraints studied in \cite{CherSonTouz}.
\item[(iii)]
Setting the generator $g(\cdot,z,\cdot,\cdot)$ equal to $+\infty$ outside a desired subset of $\R^{1\times d}$
shows for instance that our framework is flexible enough to comprise shortselling constraints.
\end{enumerate} 
\end{remark}

\subsection{General properties}

The proof of the ensuing Lemma \ref{lemma1} can be found in \citep[Lemma 3.2]{CSTP}.
\begin{lemma}\label{lemma1}
Let $g$ be a generator satisfying \ref{pos}.
Assume further that $\mathcal A(\xi,g)\neq\emptyset$ and that for the terminal condition
$\xi$ holds $\xi^-\in L^1$. 
Then $\xi\in L^1$ and, 
for any $(Y,Z)\in\mathcal A(\xi,g)$, the control $Z$ is unique and the value process $Y$ is a
supermartingale 
such that $Y_t\ge E[\xi|\mathcal F_t]$.
Moreover, the unique canonical decomposition of $Y$ is given by
\begin{equation}\label{eq21}
Y = Y_0 + M - A\,,
\end{equation}
where $M= \int ZdW$ and $A$ is an increasing, predictable, \cadlag\, process
with  $A_0=0$.
\end{lemma}
The joint convexity of the generator $g$ immediately yields the following lemma.
\begin{lemma}\label{cor22}
Let $g$ be a generator satisfying \ref{con}. Then, for each $z\in\R^{1\times d}$, the set $\mathcal A(\xi,g,z)$ is convex. 
Furthermore, from $\mathcal A(\xi^1,g,z^1)\neq\emptyset$ and $\mathcal A(\xi^2,g,z^2)\neq\emptyset$ follows 
$\mathcal A(\xi^\lambda,g, z^\lambda)\neq\emptyset$, 
for $z^\lambda: = \lambda z^1 + (1-\lambda)z^2$ and $\xi^\lambda:=\lambda \xi^1 + (1-\lambda)\xi^2$
where $\lambda\in[0,1]$.
\end{lemma}
\begin{proof}
The first assertion is a direct implication of \ref{con}.
As to the latter, it follows from \ref{con} that $\lambda(Y^1,Z^1) + (1-\lambda)(Y^2,Z^2)\in\mathcal A(\xi^\lambda,g,z^\lambda)$
whenever $(Y^1,Z^1)$ and $(Y^2,Z^2)$ belong to $\mathcal A(\xi^1,g,z^1)$ and $\mathcal A(\xi^2,g,z^2)$, respectively.
\end{proof}
For the proof of our main existence theorem we will need an auxiliary result
concerning the stability of the set $\Theta(z)$ under convergence in $\mathcal L^2$,
given that the decomposition parts can be uniformly bounded in $\mathcal L^2$.
\begin{lemma}\label{lemma21}
For any $M>0$ and $z\in\R^{1\times d}$, the set
\begin{equation*}
\Theta_M(z) = \crl{Z \in \Theta(z) \,:\, \max\left\{\norm{\Delta}_{\mathcal L^2}, \norm{\varGamma}_{\mathcal L^2}\right\} \le M }
\end{equation*}
is closed under convergence in $\mathcal L^2$.
If a sequence $(Z^n)\subset\Theta_M(z)$ with $Z^n =z+ \int \Delta^n dt + \int\varGamma^ndW$
converges in $\mathcal L^2$ to some $Z =z+ \int \Delta dt + \int\varGamma dW$, then there is a sequence $((\tilde\Delta^n,\tilde\varGamma^n))$ 
in the asymptotic convex hull of $((\Delta^n,\varGamma^n))$ converging 
in $\mathcal L^2\times\mathcal L^2$ to $(\Delta,\varGamma$).
\end{lemma}
\begin{proof}
First observe that for $Z\in\Theta(z)$ we have 
\begin{equation*}
\abs{Z_t}^2 \le 4\brak{\abs{z}^2 +\int_0^t\abs{\Delta_s}^2ds + \bigg|\int_0^t\varGamma_s dW_s\bigg|^2} \,.
\end{equation*}
Hence, for $Z\in\Theta_M(z)$, this in turn yields 
$E[|Z_t|^2]\le  4(|z|^2 + \Vert\Delta\Vert^2_{\mathcal L^2}
+ \Vert\varGamma\Vert^2_{\mathcal L^2})\le 4(|z|^2+2M^2) :=C < \infty$,
and hence $\Theta_M(z)$ is a bounded subset of $\mathcal L^2$, since by Fubini's theorem we obtain that 
$\Vert Z\Vert_{\mathcal L^2}\le \sqrt{TC}$.
Consider a sequence
$Z^n = z+ \int \Delta^n du + \int \varGamma^ndW$ in $\Theta_M(z)$ converging in $\mathcal L^2$ to some 
process $Z$.
Since $((\Delta^n,\varGamma^n))$  
are bounded in $\mathcal L^2\times \mathcal L^2$,
we can find a sequence $(\tilde \Delta^n,\tilde\varGamma^n) \in conv\{(\Delta^n,\varGamma^n),(\Delta^{n+1},\varGamma^{n+1}),\dots\}$ 
converging in $\mathcal L^2\times \mathcal L^2$ to some $(\Delta,\varGamma) \in \mathcal L^2\times \mathcal L^2$. 
Furthermore, it holds that $\Vert \Delta\Vert_{\mathcal L^2}\vee \Vert \varGamma\Vert_{\mathcal L^2}\le M$.
Let us denote by $(\tilde Z^n)$ 
the respective sequence in the asymptotic convex hull of $(Z^n)$.
From Jensen's inequality we deduce that  
\begin{equation*} 
E\edg{\int_0^T\bigg|\int_0^t(\tilde \Delta^n_s - \Delta_s)ds\bigg|^2dt}\le T E\edg{\int_0^T\abs{\tilde \Delta^n_s - \Delta_s}^2ds} \to 0\,,
\end{equation*} 
and thus $(\int\tilde \Delta^nds)$ converges to $\int\Delta ds$ in $\mathcal L^2$.
Applying Fubini's theorem and using the It\^o isometry yield that
\begin{equation*} 
E\edg{\int_0^T \bigg|\int_0^t \tilde\varGamma^n_sdW_s - \int_0^t\varGamma_s dW_s\bigg|^2 dt} \le T E\edg{\int_0^T\abs{\tilde\varGamma^n_s - \varGamma_s}^2ds}\,,
\end{equation*} 
where the term on the right-hand side tends to zero by means of the $\mathcal L^2$-convergence of $(\tilde\varGamma^n)$ to $\varGamma$.
Hence, ($\int\tilde\varGamma^ndW$) converges to $\int\varGamma dW$ in $\mathcal L^2$.
$(\tilde Z^n)$ inheriting the $\mathcal L^2$-convergence to $Z$ from $(Z^n)$ together with 
the $P\otimes dt$-uniqueness of $\mathcal L^2$-limits finally allows us to write the process $Z$ as $Z =z+ \int \Delta ds + \int \varGamma dW$,
we are done.
\end{proof} 
Lemma \ref{lemma21} yields the following compactness result.
\begin{lemma}\label{L2compact}
Assume that $\mathcal A(\xi,g,z)$ is non-empty for some $z\in\R^{1\times d}$.
Let $\xi^-$ be in $L^1$ and $g$ satisfy 
\ref{pos}, \ref{con} and \ref{dgc}.
Then, for any sequence $((Y^n,Z^n))\subset \mathcal A(\xi,g,z)$ of supersolutions satisfying $\sup_nY^n_0< \infty$, 
the following holds:
There is a sequence $(\tilde Z^n)$ in the asymptotic convex hull of $(Z^n)$ that converges in $\mathcal L^2$ to some 
process $\hat Z \in \Theta(z)$.
\end{lemma}
\begin{proof}
{\it Step 1: Existence of $((\tilde Y^n,\tilde Z^n))$.~$\mathcal L^2$-convergence of $(\tilde Z^n)$ to $\hat Z$.} 
First observe that \eqref{eq03} and the supermartingale property of all $\int Z^n dW$ imply that
\begin{equation}\label{L2C01}
E\edg{\int_0^T g_t(Y^n_t,Z^n_t,\Delta^n_t,\varGamma^n_t) dt}\le Y^n_0 + E\edg{\xi^-} \le C + E\edg{\xi^-}<\infty\,,
\end{equation}
where we put $C:=\sup_n Y^n_0$. 
Now, using \eqref{L2C01} together with \ref{dgc} we estimate
\begin{multline*}
\norm{\Delta^n}^2_{\mathcal L^2} + \norm{\varGamma^n}^2_{\mathcal L^2} = E\edg{\int_0^T \abs{\Delta^n_t}^2dt} + E\edg{\int_0^T \abs{\varGamma^n_t}^2dt} \\
\le  \frac{1}{c_2}E\edg{\int_0^T g_t(Y^n_t,Z^n_t,\Delta^n_t,\varGamma^n_t) dt} - \frac{c_1}{c_2}T 
\le \frac{1}{c_2}\brak{ C + E\edg{\xi^-} -  c_1T} <\infty\,.
\end{multline*}
Since the right-hand above is independent of $n$, 
we obtain that
$(Z^n)\subset \Theta_M(z)$ with $M:=[\frac{1}{c_2}(C + E\edg{\xi^-} -  c_1T)]^{\frac{1}{2}}$ and 
the arguments within the proof of Lemma \ref{lemma21} show that 
the sequence $(Z^n)$ is uniformly bounded in $\mathcal L^2$.
This in turn guarantees the existence of a sequence $(\tilde Z^n)$
in the asymptotic convex hull of $(Z^n)$ that converges to some process $\hat Z$ in $\mathcal L^2$ 
and, up to a subsequence, $P\otimes dt$-almost everywhere.\\

\noindent{\it Step 2: The process $\hat Z$ belongs to $\Theta(z)$.} 
The sequence $((\tilde Y^n,\tilde Z^n))$ lies in $\mathcal A(\xi,g,z)$, due to \ref{con}.
Moreover, the linearity of the integrals within the It\^o decompositions of $(Z^n)$ yields that
$\tilde Z^n = z + \int \tilde\Delta^n du + \int\tilde\varGamma^ndW$
where $((\tilde\Delta^n,\tilde\varGamma^n))$ denotes the corresponding convex combinations of the decomposition parts.
In addition, $((\tilde\Delta^n,\tilde\varGamma^n))$ inherits the uniform bound from $((\Delta^n,\varGamma^n))$, that is
$\max\{\sup_n\Vert\tilde\Delta^n\Vert_{\mathcal L^2}, \sup_n\Vert\tilde \varGamma^n\Vert_{\mathcal L^2}\} \le M$. 
Hence, Lemma \ref{lemma21} ensures that $\hat Z$ is of the form 
\begin{equation*}
\hat Z = z + \int\hat\Delta du + \int\hat\varGamma dW\,,
\end{equation*}
with suitable $\mathcal L^2$-convergence of the decomposition parts by possibly passing to yet another subsequence in the respective
asymptotic convex hull.
This finishes the proof.
\end{proof}

\subsection{Minimality under constraints}
Within the current setup of admissible controls constrained to follow certain dynamics,
we are interested in supersolutions $(\hat Y,\hat Z)\in \mathcal A(\xi,g,z)$ minimal at time zero, that is $\hat Y_0\le Y_0$
for all $(Y,Z)\in\mathcal A(\xi,g,z)$.
In the remainder of this work, a major role is thus played by the operator 
\begin{equation}\label{eq00}
\mathcal E^g_0(\xi,z) := \inf \crl{Y_0  \,:\, (Y,Z) \in \mathcal A(\xi,g,z)}\,,
\end{equation}
since any $(\hat Y,\hat Z)$ satisfying $\hat Y_0=\mathcal E^g_0(\xi,z)$ 
naturally exhibits the property of being minimal at time zero.
Note that the definition of a supersolution directly yields that $\mathcal A(\xi^1,g,z)\subseteq \mathcal A(\xi^2,g,z)$ whenever 
$\xi^1\ge\xi^2$.
Hence, we immediately obtain monotonicity of the operator $\mathcal E_0(\cdot,z)$,
that is $\xi^1 \ge \xi^2$ implies
$\mathcal E^g_0(\xi^1,z) \ge \mathcal E^g_0(\xi^2,z)$.
The ensuing Theorem \ref{prop1} provides existence of supersolutions minimal at time zero 
making use of the fact that the set $ \{Y_0 : (Y,Z) \in \mathcal A(\xi,g,z)\}$ is directed downwards. 
Parts of it rely on a version of Helly's theorem which we state here for the sake of completeness. 
In order to keep this work self-contained, we include the proof given in \citep[Lemma 1.25]{gregor_diss}.
\begin{lemma}\label{helly}
Let $(A^n)$ be a sequence of increasing positive processes such that the
sequence $(A^n_T)$ is bounded in $L^1$. 
Then, there is a sequence $(\tilde A^n)$ in the asymptotic
convex hull of $(A^n)$ and an increasing positive integrable process $\tilde A$
such that
\begin{equation*}
\lim_{n\to\infty} \tilde A^{n}_t = \tilde A_t\,,\quad \text{for all $t\in[0,T]$,\enspace $P$-almost surely}\,.
\end{equation*}
\end{lemma}
\begin{proof}
Let $(t_j)$ be a sequence running through $I := ([0,T] \cap \mathbb Q) \cup \{T\}$. 
Since $(A^n_{t_1})$ is an $L^1$-bounded sequence of positive random variables, due to \citep[Lemma A1.1]{DS94} 
there exists a sequence $(\tilde A^{1,k})$ in the asymptotic convex hull of $(A^n)$
and a random variable $\tilde A_{t_1}$ such that $(\tilde A^{1,k}_{t_1})$ converges $P$-almost surely to $\tilde A_{t_1}$.
Moreover, Fatou's lemma yields $\tilde A_{t_1} \in L^1$.
Let $(\tilde A^{2,k})$ be a sequence in the asymptotic convex hull of $(\tilde A^{1,k})$ such that
$(\tilde A^{2,k}_{t_2})$ converges $P$-almost surely to $\tilde A_{t_2}\in L^1$ and so on.
Then, for $s\in I$, it holds $\tilde A^{k,k}_s\to\tilde A_s$ on a set $\hat\Omega\subset\Omega$ satisfying $P(\hat\Omega)=1$.
The process $\tilde A$ is positive, increasing and integrable on $I$.
Thus we may define 
\begin{equation*}
\hat A_t := \lim_{r\downarrow t, r\in I} \tilde A_r\,, \quad t\in[0,T)\,,\quad \hat A_T := \tilde A_T\,.
\end{equation*}
We now show that $(\tilde A^{k,k})$, henceforth named $(\tilde A^k)$, converges $P$-almost surely on the continuity points of $\hat A$.
To this end, fix $\omega\in\hat\Omega$ and a continuity point $t\in[0,T)$ of $\hat A(\omega)$.
We show that $(\tilde A^k_t(\omega))$ is a Cauchy sequence in $\R$.
Fix $\varepsilon>0$ and set $\delta=\frac{\varepsilon}{11}$. 
Since $t$ is a continuity point of $\hat A(\omega)$, we may choose $p_1,p_2\in I$ such that $p_1<t<p_2$ and 
$\hat A_{p_1}(\omega)-\hat A_{p_2}(\omega)<\delta$.
By definition of $\hat A$, we may choose $r_1,r_2\in I$ such that $p_1<r_1<t<p_2<r_2$ and 
$|\hat A_{p_1}(\omega)-\tilde A_{r_1}(\omega)|<\delta$ and $|\hat A_{p_2}(\omega)-\tilde A_{r_2}(\omega)|<\delta$.
Now choose $N\in\N$ such that $|\tilde A^m_{r_1}(\omega)-\tilde A^n_{r_1}(\omega)|<\delta$, for all
$m,n\in\N$ with $m,n\ge N$, and $|\tilde A^j_{r_2}(\omega)-\tilde A_{r_2}(\omega)|<\delta$ and 
$|\tilde A_{r_1}(\omega)-\tilde A^j_{r_1}(\omega)|<\delta$ for $j=m,n$.
We estimate 
\begin{equation*}
|\tilde A^m_t(\omega)-\tilde A^n_t(\omega)| \le |\tilde A^m_t(\omega)-\tilde A^m_{r_1}(\omega)|+
|\tilde A^m_{r_1}(\omega)-\tilde A^n_{r_1}(\omega)|+|\tilde A^n_{r_1}(\omega)-\tilde A^n_t(\omega)| \,.
\end{equation*}
For the first and the third term on the right hand side, since $\tilde A^m$ and $\tilde A^n$ are increasing,
we deduce that $|\tilde A^m_t(\omega)-\tilde A^m_{r_1}(\omega)|\le |\tilde A^m_{r_2}(\omega)-\tilde A^m_{r_1}(\omega)|$ and 
$|\tilde A^n_t(\omega)-\tilde A^n_{r_1}(\omega)|\le |\tilde A^n_{r_2}(\omega)-\tilde A^n_{r_1}(\omega)|$.
Furthermore,
\begin{multline*}
|\tilde A^j_{r_2}(\omega)-\tilde A^j_{r_1}(\omega)| \le |\tilde A^j_{r_2}(\omega)-\tilde A_{r_2}(\omega)|
+ |\tilde A_{r_2}(\omega)-\hat A_{p_2}(\omega)| \\
+ |\hat A_{p_2}(\omega)-\hat A_{p_1}(\omega)| +  |\hat A_{p_1}(\omega)-\tilde A_{r_1}(\omega)|
+ |\tilde A_{r_1}(\omega)-\tilde A^j_{r_1}(\omega)|\,,
\end{multline*}
for $j=m,n$.
Combining the previous inequalities yields $|\tilde A^m_t(\omega)-\tilde A^n_t(\omega)|\le \varepsilon$, for all $m,n\ge N$.
Hence, $(\tilde A^k(\omega))$ converges for all continuity points $t\in[0,T)$ of $\hat A(\omega)$,
for all $\omega\in\hat\Omega$.
We denote the limit by $\tilde A$.

It remains to be shown that $(\tilde A^k)$ also converges for the discontinuity points of $\hat A$.
To this end, note that $\hat A$ is \cadlag\, and adapted to our filtration which fulfills the usual conditions.
By a well-known result, see for example \citep[Proposition 1.2.26]{karatzas01},
this implies that the jumps of $\hat A$ may be exhausted by a sequence of stopping times $(\rho^j)$.
Applying once more \citep[Lemma A1.1]{DS94} iteratively on the sequences $(\tilde A^k_{\rho^j})_{k\in\N}$, $j=1,2,3\dots$,
and diagonalizing yields the result.
\end{proof}
\begin{theorem}\label{prop1}
Assume that $\mathcal A(\xi,g,z)\neq\emptyset$ for some $\xi^-\in L^1$ and $z\in\R^{1\times d}$
and let $g$ satisfy  \ref{lsc}, \ref{pos}, \ref{con} and \ref{dgc}. 
Then, 
the set $\{(\hat Y,\hat Z) \in \mathcal A(\xi,g,z):\hat Y_0=\mathcal E_0^g(\xi,z)\}$ 
is non-empty.
\end{theorem}
\begin{proof}
\emph{Step 1: The candidate control $\hat Z$.} We extract a sequence $((Y^n,Z^n))\subset \mathcal A(\xi,g,z)$  such that
\begin{equation*}
\lim_{n\to\infty}Y^n_0 = \mathcal E_0^g(\xi,z)\,. 
\end{equation*}
Because $\sup_nY^n_0\le Y^1_0<\infty$, 
Lemma \ref{L2compact} assures the existence of a sequence $(\tilde Z^n)$ in the asymptotic convex hull 
of $(Z^n)$ that converges in $\mathcal L^2$ to some admissible process $\hat Z \in \Theta(z)$,
including $\mathcal L^2$-convergence of the corresponding decomposition parts.
In particular, we obtain that
\begin{equation}\label{eqprop02}
\int_0^t \tilde Z^n_udW_u \underset{n\to\infty}{\longrightarrow} \int_0^t\hat Z_udW_u\,, 
\quad\text{for all $t\in[0,T]$}\,,\,\, P\text{-almost surely}\,.
\end{equation}
Moreover, up to a subsequence, $((\tilde Z^n,\tilde\Delta^n,\tilde\varGamma^n))$ converges $P \otimes dt$-almost everywhere  
towards $(\hat Z,\hat\Delta,\hat\varGamma)$.\\

\noindent\emph{Step 2: The candidate value process $\hat Y$.} If we denote by $(\tilde Y^n)$ the sequence in the asymptotic convex hull of $(Y^n)$ corresponding to $(\tilde Z^n)$,
then all $(\tilde Y^n,\tilde Z^n)$ satisfy \eqref{eq03} due to \ref{con}. 
Let $\tilde A^n$ denote the increasing, predicable process of finite variation stemming from the decomposition of 
$\tilde Y^n = \tilde Y^n_0 + \tilde M^n - \tilde A^n$ given in Lemma \ref{lemma1}.
Since $(\tilde Z^n)$ is uniformly bounded in $\mathcal L^2$ and thus all $\int \tilde Z^n dW$ are true martingales, 
and $g$ satisfies \ref{pos}, the decomposition \eqref{eq21} yields
\begin{equation*}
E\edg{\tilde A^n_T} \le  Y^1_0 + E\edg{\xi^-} < \infty\,,
\end{equation*} 
as we assumed $\xi^-$ to be an element of $L^1$.
Now a version of Helly's theorem, see Lemma \ref{helly}, yields the existence of
a sequence in the asymptotic convex hull of $(\tilde A^n)$, again denoted by the previous expression, and of an increasing positive integrable process $\tilde A$
such that $\lim_{n\to\infty} \tilde A^{n}_t = \tilde A_t$, for all $t\in [0,T]$, $P$-almost surely.
We pass to the corresponding sequence on the side of $(\tilde Y^n)$ and $(\tilde Z^n)$,
define the process $\tilde Y$ pointwise for all $t\in[0,T]$ by
$\tilde Y_t := \lim_{n\to\infty} \tilde Y^n_t 
=  \mathcal E^g_0(\xi,z) + \int_0^t\hat Z_udW_u - \tilde A_t$,
and observe that it fulfills $\tilde Y_0 = \mathcal E_0^g(\xi,z)$ by construction.
However, since $\tilde Y$ is not necessarily \cadlag, we define our candidate value process $\hat Y$ by
$\hat Y_t := \lim_{s\downarrow t, s\in\mathbb Q} \tilde Y_s$, 
for all $t\in[0,T)$ and $\hat Y_T := \xi$.
The continuity of $\int \hat ZdW$ yields that 
\begin{equation}\label{eqprop04}
\hat Y_t = \mathcal E^g_0(\xi,z) + \int_0^t \hat Z_udW_u -  \lim_{s\downarrow t, s\in\mathbb Q} \tilde A_s\,.
\end{equation}
Since jump times of \cadlag\, processes\footnote{Note that as an increasing process, $\tilde A$ is in particular a submartingale and thus
its right- and left-hand limits exist, compare \citep[Proposition 1.3.14]{karatzas01}.%
~Consequently, the process $\lim_{s\downarrow \cdot, s\in\mathbb Q} \tilde A_s$ is \cadlag.} can be exhausted by a 
sequence of stopping times $(\sigma_j)\subset \mathcal T$, 
compare \citep[Proposition 1.2.26]{karatzas01}, which coincide with the jump times of $\tilde A$, we conclude that 
\begin{equation}\label{eqprop04a}
\hat Y=\tilde Y\,,\quad \text{$P\otimes dt$-almost everywhere}\,.
\end{equation}
Furthermore, $\tilde A$ increasing implies that
$\hat A_t := \lim_{s\downarrow t, s\in\mathbb Q} \tilde A_s \ge \tilde A_t$, for all $t\in[0,T]$ 
which, together with \eqref{eqprop04}, in turn yields that 
\begin{equation}\label{eqprop05}
\hat Y_t \le \tilde Y_t\,,\quad \text{for all $t\in[0,T]$}\,.
\end{equation}
Given that $(\hat Y,\hat Z)$ satisfies \eqref{eq03}, 
we could 
conclude that $(\hat Y,\hat Z)\in\mathcal A(\xi,g,z)$ 
and thus $\hat Y_0 \ge \mathcal E^g_0(\xi,z)=\tilde Y_0$ which, 
combined with \eqref{eqprop05}, 
would imply $\hat Y_0 = \mathcal E^g_0(\xi,z)$ 
and thereby finish the proof.\\

\noindent\emph{Step 3: Verification.} As to the remaining verification, 
we deduce from \eqref{eqprop04a} the existence of a set $A\in\mathcal F_T$, $P(A)=1$ 
with the following property.
For all $\omega \in A$, there exists a Lebesgue measurable set $\mathcal I(\omega) \subset [0,T]$
of measure $T$ such that  $\tilde Y^n_t(\omega) \longrightarrow \hat Y_t(\omega)$, for all $t\in \mathcal I(\omega)$.
We suppress the dependence of $\mathcal I$ on $\omega$ and recall however that in the following $s$ and $t$ 
may depend on $\omega$.
For $s,t\in\mathcal I$ with $s\le t $ holds
\begin{multline}\label{eqprop07}
\hat Y_s  - \int^t_s g_u(\hat Y_u,\hat Z_u, \hat\Delta_u,\hat\varGamma_u)du + \int_s^t \hat Z_udW_u\\
\ge \limsup_n\brak{\tilde Y^{n}_s  - 
\int_s^t g_u(\tilde Y^n_u,\tilde Z^{n}_u,\tilde\Delta^{n}_u,\tilde\varGamma^{n}_u)du + \int_s^t \tilde Z^{n}_udW_u}
\end{multline}
by means of \eqref{eqprop02}, the $P\otimes dt$-almost-everywhere 
convergence of $((\tilde Y^n,\tilde Z^n,\tilde\Delta^{n},\tilde\varGamma^{n}))$ towards $(\hat Y,\hat Z,\hat \Delta,\hat\varGamma)$,
the property \ref{lsc} and Fatou's lemma.
Using $((\tilde Y^n,\tilde Z^n))\subset \mathcal A(\xi,g,z)$, for all $n\in\N$, 
\eqref{eqprop07} can be further estimated by
\begin{equation}\label{eqprop08}
\hat Y_s  - \int_s^t g_u(\hat Y_u,\hat Z_u, \hat\Delta_u,\hat\varGamma_u)du + \int_s^t \hat Z_udW_u
\ge \limsup_n\tilde Y^n_t = \hat Y_t\,.
\end{equation}
Whenever $s,t\in\mathcal I^c$ with $s\le t$, we approximate both times from the right by sequences 
$(s^n)\subset\mathcal I$ and $(t^n)\subset \mathcal I$, respectively, such that $s^n\le t^n$.
Since \eqref{eqprop08} holds for all $s^n$ and $t^n$, the claim follows from the right-continuity of $\hat Y$ and the continuity
of all appearing integrals, which finally concludes the proof.
\end{proof}
Convexity of the mapping $(\xi,z)\mapsto \mathcal E_0^g(\xi,z)$
is provided by the following lemma.
\begin{lemma}\label{lem_convex}
Under the assumptions of Theorem \ref{prop1}, the operator $\mathcal E_0^g(\cdot,\cdot)$ is jointly convex. 
\end{lemma}
\begin{proof}
For $z^1,z^2\in\R^{1\times d}$ and $\xi^1,\xi^2\in L^0$, the negative parts of which are integrable, 
assume that $\mathcal A(\xi^1,g,z^1)\neq\emptyset$ and $\mathcal A(\xi^2,g,z^2)\neq\emptyset$, as otherwise convexity
trivially holds.
For $\lambda\in[0,1]$ we set $z^\lambda: = \lambda z^1 + (1-\lambda)z^2$ and $\xi^\lambda:=\lambda \xi^1 + (1-\lambda)\xi^2$
so that Lemma \ref{cor22} implies  $\mathcal A(\xi^\lambda,g,z^\lambda)\neq\emptyset$.
By Theorem \ref{prop1}, there exist $(Y^1,Z^1)$ and $(Y^2,Z^2)$ in $\mathcal A(\xi^1,g,z^1)$ and $\mathcal A(\xi^2,g,z^2)$, respectively, such that
$ Y^1_0 = \mathcal E_0^g(\xi^1,z^1)$ and $Y^2_0 =\mathcal E_0^g(\xi^2,z^2)$.
Since $(\bar Y,\bar Z) := \lambda (Y^1,Z^1)  + (1-\lambda)(Y^2,Z^2)$ is an element of $\mathcal A(\xi^\lambda,g,z^\lambda)$ due to \ref{con},
it holds $\mathcal E_0^g(\xi^\lambda,z^\lambda)\le \bar Y_0 $ by definition of the operator $\mathcal E^g_0$.
\end{proof}

\subsection{Stability results}
Next, we show that the non-linear operator $\xi\mapsto \mathcal E_0^g(\xi,z)$ 
exhibits stability properties such as monotone convergence, the Fatou property or $L^1$-lower semicontinuity.
The following theorem establishes monotone convergence and the Fatou property of $\mathcal E^g_0(\cdot,z)$.
Similar results in the unconstrained case have been obtained in \citep[Theorem 4.7]{CSTP}.
\begin{theorem}\label{mon_fatou}
For $z\in\R^{1\times d}$ and $g$ a generator fulfilling \ref{lsc}, \ref{pos}, \ref{con} and \ref{dgc}, and $(\xi_n)$ a sequence 
in $L^0$ such that $(\xi_n^-)\subset L^1$, 
the following holds.
\begin{itemize}
\item Monotone convergence: If $(\xi_n)$ is increasing $P$-almost surely to $\xi\in L^0$, then it holds
$\lim_{n\to\infty} \mathcal E_0^g(\xi_n,z) = \mathcal E_0^g(\xi,z)$.
\item Fatou's lemma: If $\xi_n\ge\eta$, for all $n\in\N$, where $\eta\in L^1$, then it holds 
$\mathcal E_0^g(\liminf_n \xi_n,z) \le \liminf_n \mathcal E_0^g(\xi_n,z)$.
\end{itemize}
\end{theorem}
\begin{proof}
\textit{Monotone convergence:} First, note that by monotonicity 
the limit $\bar Y_0:=\lim_n\mathcal E_0^g(\xi_n,z)$ exists and satisfies $\bar Y_0\le\mathcal E_0^g(\xi,z)$.  
Other than in the trivial case of $+\infty=\bar Y_0\le\mathcal E_0^g(\xi,z)$ 
we have $\mathcal A(\xi_n,g,z)\neq\emptyset$, for all $n\in\N$, which, together with $(\xi_n^-)\subset L^1$
implies $(\xi_n)\subset L^1$.
Furthermore, Theorem \ref{prop1} yields the existence of supersolutions $(Y^n,Z^n)\in\mathcal A(\xi_n,g,z)$ 
fulfilling $Y^n_0=\mathcal E_0^g(\xi_n,z)$, for all $n\in\N$.
In particular, we have that $Y^n_0\le \bar Y_0$ and $\xi^-_n\le \xi^-_1$, for all $n\in\N$.
Arguments analogous to the ones used in Lemma \ref{L2compact} and the proof of Theorem \ref{prop1} directly translate to the present setting and 
provide both a candidate control $\hat Z\in\Theta(z)$
to which $(\tilde Z^n)$ converges and a corresponding $\tilde Y_t:=\lim_n\tilde Y^n_t$, 
and ensure that $(\hat Y,\hat Z)$ belongs to $\mathcal A(\xi,g,z)$,
where $\hat Y:=\lim_{s\in\mathbb Q,s\downarrow\cdot}\tilde Y_s$ on $[0,T)$ and $\hat Y_T:=\xi$. 
In particular, we obtain 
$\hat Y_0\le \tilde Y_0=\bar Y_0$.
Hence, as $\mathcal A(\xi,g,z)\neq\emptyset$ and $\xi^-\in L^1$, 
there exists $(Y,Z)\in\mathcal A(\xi,g,z)$ such that $Y_0=\mathcal E_0^g(\xi,z)$. 
By minimality of $(Y,Z)$ at time zero, however, this entails $Y_0\le\hat Y_0\le \bar Y_0$ and we conclude that 
$\lim_{n\to\infty}\mathcal E_0^g(\xi_n,z) = \mathcal E_0^g(\xi,z)$.

\textit{Fatou's lemma:} If we define $\zeta_n:=\inf_{k\ge n}\xi_k$, then $\xi_k\ge\eta$ for all $k\in\N$ 
implies $\zeta_n\ge\eta$ for all $n\in\N$ which in turn gives $(\zeta^-_n)\subset L^1$,
and thus the monotone convergence established above can be used exactly as in \citep[Theorem 4.7]{CSTP} to obtain the assertion.
\end{proof}
As a consequence of the monotone convergence property we obtain the ensuing theorem providing $L^1$-lower semicontinuity of the 
operator $\mathcal E^g_0(\cdot,z)$.
The proof goes along the lines of \citep[Theorem 4.9]{CSTP}
and is thus omitted here.
\begin{theorem}\label{thm_L1_lsc}
Let $z\in\R^{1\times d}$ and $g$ be a generator fulfilling \ref{lsc}, \ref{pos}, \ref{con} and \ref{dgc}.
Then $\mathcal E^g_0(\cdot,z)$ is $L^1$-lower semicontinuous. 
\end{theorem}

\section{Duality under constraints}\label{sec023}
The objective of this section is to 
construct a solution of constrained BSDEs via duality
and, for the case of a quadratic generator, to obtain an explicit form for $\mathcal E^*_0$, the Fenchel-Legendre transform of $\mathcal E^g_0$.
Let us assume for the rest of this section that our generator $g$ is independent of $y$, that is
$g_u(y,z,\delta,\gamma)=g_u(z,\delta,\gamma)$, and that it satisfies \ref{lsc}, \ref{pos}, \ref{con} and \ref{dgc}.
Let us further fix some $z\in \R^{1\times d}$ as initial value of the controls and 
set $\mathcal E^g_0(\cdot):=\mathcal E^g_0(\cdot,z)$ 
for the remainder of this section.
Whenever we say that the BSDE$(\xi,g)$ has a solution $(Y,Z)$, we mean that there exists $(Y,Z)\in\mathcal A(\xi,g,z)$ such that
\eqref{eq03} is satisfied with equalities instead of inequalities.
Observe that $\mathcal E^g_0(\cdot)$, being convex and $L^1$-lower semicontinuous, is in particular $\sigma(L^1,L^\infty)$-lower 
semicontinuous, and thus, by classical duality results admits the Fenchel-Moreau representation
\begin{equation}\label{eq_duality_02}
\mathcal E^g_0(\xi) = \sup_{v\in L^\infty} \crl{E[v\xi] - \mathcal E_0^*(v)}\,,\quad \xi\in L^1\,,
\end{equation}
where for $v\in L^\infty$ the convex conjugate is given by
\begin{equation*}
\mathcal E^*_0 (v):= \sup_{\xi\in L^1} \crl{E[v\xi] - \mathcal E_0^g(\xi)} \,.
\end{equation*}
It is proved in the next lemma that the domain of $\mathcal E^*_0$ is concentrated on non-negative $v\in L^\infty_+$
satisfying $E[v]=1$.
\begin{lemma}\label{lemma_duality_appendix}
Within the representation \eqref{eq_duality_02}, that is $\mathcal E^g_0(\xi) = \sup_{v\in L^\infty} \{E[v\xi] - \mathcal E_0^*(v)\}$, 
the supremum might be restricted to those $v\in L^\infty_+$ satisfying $E[v]=1$.
\end{lemma}
\begin{proof}
First, we assume without loss of generality that $\mathcal E^g_0(0)<+\infty$.
Indeed, a slight modification of the argumentation below remains valid 
using any $\xi\in L^1$ such that $\mathcal E^g_0(\xi)<+\infty$.\footnote{Note that the case $\mathcal E^g_0\equiv+\infty$ on $L^1$
immediately yields $\mathcal E^*_0\equiv-\infty$ on $L^\infty$ and is thus neglected.}
We show that $\mathcal E^*_0(v)=+\infty$ as soon as $v\in L^\infty\backslash L^\infty_+$ or $E[v]\neq 1$. 
For $v\in L^\infty\backslash L^\infty_+$, $L^1_+$ being the polar of $L^\infty_+$ yields the existence of $\bar\xi\in L^1_+$ such that 
$E[v\bar\xi]<0$.
Monotonicity of $\mathcal E^g_0$ then gives $\mathcal E^g_0(-n\bar\xi) \le \mathcal E^g_0(0)$ for all $n\in\N$.
Hence,
\begin{equation*}
\mathcal E^*_0(v) \ge \sup_n\crl{nE[-v\bar\xi]-\mathcal E^g_0(-n\bar\xi)} \ge \sup_n\crl{nE[-v\bar\xi]} -\mathcal E^g_0(0) = + \infty\,.
\end{equation*}
Furthermore, since the generator does not depend on $y$, the function $\mathcal E^g_0$ is cash additive, 
compare \citep[Proposition 3.3.5]{CSTP}, and we deduce that, 
for all $n\in\N$ it holds
\begin{equation*}
\mathcal E^*_0(v) \ge E[vn] - \mathcal E^g_0(0) - n = n(E[v]-1) - \mathcal E^g_0(0)\,. 
\end{equation*}
Thus, if $E[v]>1$, then $\mathcal E^*_0(v)=+\infty$.
A reciprocal argument with $\xi=-n$ finally gives $\mathcal E^*_0(v)=+\infty$ whenever $E[v]<1$.
\end{proof}
By the previous result,  we may use $v\in L^\infty_+$, $E[v]=1$, in order to define a measure $Q$ that is absolutely continuous with respect to $P$
by setting $\frac{dQ}{dP} := v$.
Thereby \eqref{eq_duality_02} may be reformulated as 
\begin{equation}\label{eq_duality_03b}
\mathcal E^g_0(\xi) = \sup_{Q\ll P} \crl{E_Q[\xi] - \mathcal E_0^*(Q)}\,, \quad\xi\in L^1\,,
\end{equation}
where
\begin{equation}\label{eq_duality_03a}
\mathcal E^*_0 (Q):= \sup_{\xi\in L^1} \crl{E_Q[\xi] - \mathcal E_0^g(\xi)}\,.
\end{equation}
Note that $\mathcal A(\xi,g,z)=\emptyset$ implies $\mathcal E^g_0(\xi)=+\infty$ and hence such terminal
conditions are irrelevant for the supremum in \eqref{eq_duality_03a}.
Let us denote by $\mathcal Q$ the set of all probability measures equivalent to $P$ with bounded Radon-Nikodym derivative.
For each $Q\in\mathcal Q$, there exists a progressively measurable process $q$ taking values in $\R^{1\times d}$ such that for all $t\in[0,T]$
\begin{equation*}
\frac{dQ}{dP} |_{\mathcal F_t} = \exp\brak{\int_0^t q_udW_u - \frac{1}{2} \int_0^t |q_u|^2 du} \,.
\end{equation*}
By Girsanov's theorem, the process $W^Q_t := W_t - \int_0^t q_u du$ is a $Q$-Brownian motion.
The following lemma is a valuable tool regarding the characterization of $\mathcal E^*_0$.
\begin{lemma}\label{lem_dual00}
The supremum in \eqref{eq_duality_03a} can be restricted to random variables $\xi \in L^1$
for which the BSDE with parameters $(\xi,g)$ has a solution with value process starting in $\mathcal E^g_0(\xi)$. 
More precisely, for any $Q \in \mathcal{Q}$ holds
\begin{equation*}\label{eq:proof02}
\mathcal{E}^\ast_0(Q) = \sup_{ \xi \in L^1}\{ E_Q\left[ \xi \right] - \mathcal{E}^g_0(\xi) \,: \, 
\mbox{BSDE$(\xi, g)$ has a solution $(Y,Z)$ with $Y_0=\mathcal E^g_0(\xi)$}\}\,. 
\end{equation*}
\end{lemma}
\begin{proof}
It suffices to show that
\begin{multline}\label{solution}
\mathcal{E}^\ast_0(Q) \\
\le \sup_{\xi\in L^1}\{ E_{Q}[\xi] - \mathcal{E}_0^g(\xi) \,:\, \text{BSDE$(\xi,g)$ has a solution $(Y,Z)$ with $Y_0=\mathcal E^g_0(\xi)$} \}\,,
\end{multline}
since the reverse inequality is satisfied by definition of $\mathcal E^*_0(\cdot)$.
Consider to this end a terminal condition $\xi \in L^1$ with associated minimal supersolution $(Y,Z)\in\mathcal A(\xi,g,z)$, that is $Y_0 = \mathcal E^g_0(\xi)$. 
Put, for all $t \in [0,T]$,
\begin{equation*}
Y^1_t=\mathcal{E}^g_0( \xi ) - \int_{0}^{t}g_u\left( Z_u,\Delta_u,\varGamma_u \right)du+\int_{0}^{t}Z_u dW_u\,.
\end{equation*}
Relation \eqref{eq03} implies $Y^1_T\geq Y_T \geq \xi$ and thus $\mathcal{E}^g_0(Y^1_T)\geq \mathcal{E}^g_0(\xi)$ and $(Y^1_T)^- \in L^1$.
Furthermore, observe that 
\begin{equation*}
(Y^1_T)^+ = \brak{\mathcal{E}^g_0( \xi ) - \int_0^Tg_u\left(Z_u,\Delta_u,\varGamma_u \right)du+\int_0^TZ_u dW_u}^+ \le
\brak{\mathcal{E}^g_0( \xi )  +\int_0^TZ_u dW_u}^+
\end{equation*}
due to the positivity of the generator.
But since the right-hand side is in $L^1$ by means of the martingale property of $\int Z dW$, we deduce that 
$(Y^1_T)^+\in L^1$, allowing us to conclude that $Y^1_T\in L^1$.
On the other hand, $(Y^1,Z)\in \mathcal{A}(Y^1_T,g,z)$ holds by definition of $Y^1$.
Hence, we conclude $\mathcal{E}^g_0(Y^1_T) \leq Y^1_0 =\mathcal{E}^g_0(\xi)$. 
Thus, $\mathcal{E}^g_0(Y^1_T) = \mathcal{E}^g_0(\xi)$, and $(Y^1,Z)$ is a solution of the BSDE with parameters $(Y^1_T,g)$. 
Observe further that $\mathcal{E}^g_0(Y^1_T) - \mathcal{E}^g_0(\xi)=0\le Y^1_T - \xi$
which, by taking expectation under $Q$, implies
\begin{equation*}
E_{Q}[\xi] - \mathcal{E}^g_0(\xi) \leq E_{Q}[Y^1_T ] - \mathcal{E}^g_0(Y^1_T)\,.
\end{equation*}
Taking the supremum yields \eqref{solution}, the proof is done.
\end{proof}
By means of the preceding lemma it holds
\begin{align}\label{eq_duality_01}
\mathcal E^*_0 (Q)= &\sup_{\xi\in L^1} \crl{E_Q[\xi] - \mathcal E_0^g(\xi)} \nonumber\\
= &\sup_{\xi\in L^1} \crl{E_Q\left[\mathcal E_0^g(\xi) - \int_0^T g_u(Z_u,\Delta_u,\varGamma_u) du + 
         \int_0^T Z_u dW_u\right] - \mathcal E_0^g(\xi)} \nonumber\\
= &\sup_{(\Delta,\varGamma)\in\Pi} \crl{E_Q\left[- \int_0^T  g_u(Z_u,\Delta_u,\varGamma_u)  du + 
         \int_0^T Z_u dW_u\right] }         
\end{align}
where 
\begin{equation}\label{eq_Pi}
\Pi:=\crl{(\Delta,\varGamma)\in\mathcal L^2\times\mathcal L^2 :\begin{array}{l}  \exists \xi\in L^1:\enspace\mbox{BSDE$(\xi,g)$ has a solution $(Y,Z)$}\\
\mbox{with $Y_0=\mathcal E^g_0(\xi)$ and $Z = z+\int\Delta du + \int\varGamma dW$}  \end{array} } \,.
\end{equation}
Whenever $Q\in\mathcal Q$, Girsanov's theorem applies and we may exploit the decomposition of $Z$ and use 
that $\int Z dW^Q$ and $\int \varGamma dW^Q$are $Q$-martingales
in order to express the right-hand side of \eqref{eq_duality_01}  without Brownian integrals. 
More precisely,
\begin{multline}\label{eq_duality_05}
\mathcal E^*_0 (Q) = 
\sup_{(\Delta,\varGamma)\in\Pi} \crl{E_Q\left[ \int_0^T \brak{- g_u(Z_u,\Delta_u,\varGamma_u) 
+ q_u\int_0^u(\Delta_s + q_s\varGamma_s)ds} du\right]} \\ + zE_Q\Bigg[\int_0^Tq_udu\Bigg]\,.
\end{multline}
We continue with two lemmata that allow us to restrict the set of measures in the representation \eqref{eq_duality_03b}
to a sufficiently nice subset of $\mathcal Q$ on the one hand, and to change the set $\Pi$ appearing in \eqref{eq_duality_01}
to the whole space $\mathcal L^2\times\mathcal L^2$ on the other hand.
\begin{lemma}\label{lem_restrict_measures}
Assume there exists some $\xi \in L^1$ such that $\mathcal A(\xi,g,z)\neq\emptyset$. 
Then it is sufficient to consider measures with densities that are bounded away from zero, that is
\begin{equation}\label{eq_dual_strict_pos}
\mathcal E^g_0(\xi) = \sup_{v\in L^\infty_b} \crl{E[v\xi] - \mathcal E_0^*(v)} 
\end{equation}
where $L^\infty_b:=\{v\in L^\infty:\,v>0\quad\mbox{and}\quad \|\frac{1}{v}\|_\infty<\infty\}$. 
\end{lemma}
\begin{proof}
The assumption of $\mathcal A(\xi,g,z)$ being non-empty for some $\xi\in L^1$ implies
the existence of $(\Delta,\varGamma)\in\Pi$ and corresponding $Z$ such that $E_P[\int_0^Tg_u(Z_u,\Delta_u,\varGamma_u)du]<\infty$ which
together with \ref{pos}, \eqref{eq_duality_01} 
and the martingale property of all occurring $\int ZdW$ under $P$ immediately yields that $\mathcal E^*_0(P)<\infty$.
For any $Q\ll P$  with $\frac{dQ}{dP}=v\in L^\infty_+$ and $\lambda\in(0,1)$ we define a measure $Q^\lambda$
by its Radon-Nikodym derivative $ v_\lambda:= (1-\lambda)v + \lambda$ where 
naturally $\frac{dP}{dP}=1$.
Observe that $\lambda>0$ implies $v_\lambda\in L^\infty_b$.
Next, we show that $\lim_{\lambda\downarrow 0}\mathcal E^*_0(v_\lambda)=\mathcal E^*_0(v)$.
Indeed, convexity of $\mathcal E^*_0(\cdot)$ together with $\mathcal E^*_0(\frac{dP}{dP})=\mathcal E^*_0(1)<\infty$  
yields $\liminf_{\lambda\downarrow 0}\mathcal E^*_0(v_\lambda)\le\mathcal E^*_0(v)$, 
whereas the reverse inequality is satisfied by means of the lower semicontinuity.
On the other hand, dominated convergence gives $\lim_{\lambda\downarrow 0} E[v_\lambda\xi]=E[v\xi]$, since
$|v_\lambda\xi| \le |v\xi|+|\xi|$ which is integrable.
Consequently, the expression $\{E[v\xi] - \mathcal E^*_0(v)\}$ is the limit of a sequence 
$(E[v_{\lambda_n}\xi] - \mathcal E^*_0(v_{\lambda_n}))_n$ where $(v_{\lambda_n})\subset L^\infty_b$ and $\lambda_n\downarrow 0$.
Since $\mathcal E^g_0(\xi)$ can be expressed as the supremum of $\{E[v\xi] - \mathcal E^*_0(v)\}$ over all $v$, 
it suffices to consider the supremum over $v\in L^\infty_b$, the proof is done.
\end{proof}
\begin{lemma}\label{lem_optimization}
For each $Q\in\mathcal Q$ such that $\frac{dQ}{dP}\in L^\infty_b$
it holds 
\begin{equation}\label{eq_dual_optim}
\mathcal E^*_0 (Q) =  \sup_{(\Delta,\varGamma)\in\mathcal L^2\times\mathcal L^2} \crl{E_Q\left[- \int_0^T  g_u(Z_u,\Delta_u,\varGamma_u)  du + 
         \int_0^T Z_u dW_u\right] }\,.
\end{equation} 
\end{lemma}
\begin{proof}
Since $\Pi$ defined in \eqref{eq_Pi} is a subset of $\mathcal L^2\times\mathcal L^2$, ``$\le$'' certainly holds in \eqref{eq_dual_optim}. 
As to the reverse inequality, observe first that, since we consider a supremum in \eqref{eq_dual_optim} and $Z\in \mathcal L^2$ 
whenever $(\Delta,\varGamma)\in\mathcal L^2\times\mathcal L^2$, 
those $(\Delta,\varGamma)$ such that $E_Q[\int_0^Tg_u(Z_u,\Delta_u,\varGamma_u)du]=+\infty$ can be neglected in the following.
In particular, since $v=\frac{dQ}{dP}\in L^\infty_b$, we can restrict our focus to those elements satisfying 
$E[\int_0^Tg_u(Z_u,\Delta_u,\varGamma_u)du]\le
\|\frac{1}{v}\|_{L^{\infty}}E_Q[\int_0^Tg_u(Z_u,\Delta_u,\varGamma_u)du]<+\infty$.
Thus, given such a pair $(\Delta,\varGamma)$, the terminal condition 
$\xi:=  -\int_0^Tg_u(Z_u,\Delta_u,\varGamma_u)du + \int_0^TZ_udW_u$
fulfills $\xi^-\in L^1$ due to the martingale property of $\int Z dW$.
Furthermore, the pair $(-\int_0^\cdot g_u(Z_u,\Delta_u,\varGamma_u)du + \int_0^\cdot Z_udW_u,Z)$ is an element of $\mathcal A(\xi,g,z)$ 
by construction and hence Theorem \ref{prop1} yields the existence
of $(\bar Y,\bar Z)\in \mathcal A(\xi,g,z)$ satisfying $\bar Y_0 = \mathcal E^g_0(\xi)\le 0$.
Now, using the same techniques as in the proof of Lemma \ref{lem_dual00}, we define $Y^1$ by
$Y^1_t := \mathcal E^g_0(\xi) -\int_0^tg_u(\bar Z_u,\bar\Delta_u,\bar\varGamma_u)du + \int_0^t \bar Z_u dW_u$, for all $t\in[0,T]$,
where $(\bar\Delta,\bar\varGamma)$ is the decomposition of $\bar Z$, and obtain that 
$Y^1_T\ge\xi$ as well as $\mathcal E^g_0(Y^1_T)=\mathcal E^g_0(\xi)$.
Consequently,
\begin{multline}\label{eq_optim1}
-\int_0^T g_u(\bar Z_u,\bar\Delta_u,\bar\varGamma_u)du + \int_0^T \bar Z_u dW_u \\= Y^1_T -\mathcal E^g_0(\xi) \ge Y^1_T 
\ge \xi = -\int_0^Tg_u(Z_u,\Delta_u,\varGamma_u)du + \int_0^T Z_u dW_u \,,
\end{multline}
which, by taking expectation under $Q$ in \eqref{eq_optim1} and using $(\bar\Delta,\bar\varGamma)\in\Pi$, implies
\begin{equation}\label{eq_optim2}
\mathcal E^*_0(Q)\ge E_Q\edg{-\int_0^Tg_u(Z_u,\Delta_u,\varGamma_u)du + \int_0^T Z_u dW_u}\,.
\end{equation}
Since $(\Delta,\varGamma)$ was arbitrary, we have finally shown that $\mathcal E^*_0(Q)$ is greater or equal to 
the supremum over $(\Delta,\varGamma)\in\mathcal L^2\times\mathcal L^2$
of the right-hand side of \eqref{eq_optim2}, 
which finishes the proof.
\end{proof}
The ensuing proposition provides, for a given measure $Q\in\mathcal Q$ with $\frac{dQ}{dP}\in L^\infty_b$, the existence of a pair of processes attaining 
the supremum in \eqref{eq_duality_01}.
\begin{proposition}\label{lem_dual01}
For each $Q\in\mathcal Q$ with $\frac{dQ}{dP}\in L^\infty_b$ there exist $(\Delta^Q,\varGamma^Q)\in\Pi$ and a corresponding control $Z^Q$ of the
form $Z^Q = z + \int \Delta^Q du + \int \varGamma^Q dW$ such that 
\begin{equation}\label{eq_attaining}
\mathcal E^*_0 (Q) 
= E_Q\left[ -\int_0^T g_u(Z^Q_u,\Delta^Q_u,\varGamma^Q_u) du  + \int_0^T Z^Q_udW_u \right]\,.
\end{equation}
Furthermore, if the convexity of $g$ is strict, then the triple $(Z^Q,\Delta^Q,\varGamma^Q)$ is unique.
\end{proposition}
\begin{proof}
{\it Step 1: The integral $\int qdW$ is an element of $BMO$.}
We begin by proving that, for $Q\in\mathcal Q$ the density $\frac{dQ}{dP}=\exp(\int_0^T q_udW_u - \frac{1}{2} \int_0^T |q_u|^2 du)$
of which belongs to $L^\infty_b$, 
the process $(\int_0^tq_udW_u)_{t\in[0,T]}$ is an element of $BMO$.
Indeed, since the process $v_t:= E[\frac{dQ}{dP}\,|\,\mathcal F_t]$ is uniformly bounded away from zero,
it satisfies the Muckenhaupt $(A_1)$ condition, see \citep[Definition 2.2]{Kazamaki01}, and therefore $\int q dW \in BMO$ by means of
\citep[Theorem 2.4]{Kazamaki01}.\\

\noindent\emph{Step 2: $\mathcal L^2$-boundedness of a minimizing sequence and the candidate $(\Delta^Q,\varGamma^Q)$.}
Since the generator $g$ satisfies \ref{dgc}, it holds for all $(\Delta,\varGamma,Z)$ that
\begin{equation}\label{eq_g_quad}
\norm{\Delta}^2_{\mathcal L^2(Q)} + \norm{\varGamma}^2_{\mathcal L^2(Q)} \le \frac{1}{c_2} 
\brak{ E_Q\edg{\int_0^T g_u(Z_u,\Delta_u,\varGamma_u)du} - c_1 T}\,.
\end{equation}
If we put $F(Z,\Delta,\varGamma) := E_Q[ \int_0^T (g_u(Z_u,\Delta_u,\varGamma_u)du  - q_u\int_0^u(\Delta_s + q_s\varGamma_s)ds) du]$,  
then \eqref{eq_duality_05} in combination with Lemma \ref{lem_optimization} implies that the conjugate can be expressed
by $\mathcal E^*_0 (Q) = -\inf_{(\Delta,\varGamma)\in\mathcal L^2\times\mathcal L^2} F(Z,\Delta,\varGamma) + zE_Q[\int_0^Tq_udu]$. 
We claim that, for $(Z^n,\Delta^n,\varGamma^n)$ a minimizing sequence of $F$, both $(\Delta^n)$ and $(\varGamma^n)$ are bounded in $\mathcal L^2(Q)$.
Since in our case the $\mathcal L^2$-norms with respect to $P$ and $Q$ are equivalent, we suppress the dependence on
the measure in the notation to follow.
Assume now contrary to our assertion that $\|\Delta^n\|^2_{\mathcal L^2} \rightarrow\infty$ and
$\|\varGamma^n\|^2_{\mathcal L^2} \rightarrow\infty$
as $n$ tends to infinity.
This in turn would imply either 
\begin{equation}\label{eq_same_order1}
E_Q \edg{ \int_0^T  q_u\int^u_0\Delta^n_s ds\,du}\rightarrow \infty\qquad\mbox{and}\qquad
\limsup_n\frac{\norm{\Delta^n}_{\mathcal L^2}^2}{E_Q\edg{\int_0^T  q_u\int^u_0\Delta^n_s ds\,du}} = K 
\end{equation}
or 
\begin{equation}\label{eq_same_order2}
E_Q\edg{ \int_0^T  q_u\int^u_0q_s\varGamma^n_s ds\,du}\rightarrow \infty \,\,\quad\mbox{and}\quad\,\,
\limsup_n\frac{\norm{\varGamma^n}_{\mathcal L^2}^2}{E_Q\edg{ \int_0^T  q_u\int^u_0q_s\varGamma^n_s ds\,du}} = L
\end{equation}
or both, for $K,L\in\R$.
Indeed, \eqref{eq_g_quad} would otherwise lead to $\lim_n F(Z^n,\Delta^n,\varGamma^n)=\infty$ 
and thereby contradict $((Z^n,\Delta^n,\varGamma^n))$ being a minimizing sequence of $F$.
On the other hand however, an application of H\"older's inequality 
yields
\begin{multline}\label{eq_hoelder}
\abs{E_Q\edg{\int_0^T q_u \int_0^u\Delta^n_sds\,du}} 
\le  \brak{E_Q\edg{\int_0^T |q_u|^2 du}E_Q\edg{ \int_0^T \brak{\int_0^u\Delta^n_sds}^2 du}}^{\frac{1}{2}} \\
\le \norm{q}_{\mathcal L^2}\brak{ E_Q\edg{\int_0^T \int_0^u|\Delta^n_s|^2ds du}}^{\frac{1}{2}} 
\le T^{\frac{1}{2}} \norm{q}_{\mathcal L^2} \norm{\Delta^n}_{\mathcal L^2}\,.
\end{multline}
Taking the square on both sides above we obtain
\begin{equation*}
\brak{E_Q\left[\int_0^T q_u\int_0^u\Delta^n_s ds\,du\right]}^2
\le T \norm{q}_{\mathcal L^2}^2 \norm{\Delta^n}_{\mathcal L^2}^2
\end{equation*}
which in turn implies $\norm{\Delta^n}^2_{\mathcal L^2}(E_Q[ \int_0^T  q_u\int^u_0\Delta^n_s ds\,du])^{-1} \to \infty$, 
a contradiction to \eqref{eq_same_order1}.
As to $(\varGamma^n)$, we argue similarly and, for $(Q_u)_{u\in[0,T]}$ defined by $Q_u:=\int_u^Tq_sds$ estimate 
\begin{multline*}
\abs{E_Q\edg{\int_0^T q_u \int_0^u q_s\varGamma^n_sds\,du}} = \abs{ E_Q\left[ \int_0^Tq_u\varGamma^n_uQ_u du\right] }
\le E_Q\left[ \int_0^T|q_u||\varGamma^n_u||Q_u| du\right]\\
\le  \brak{E_Q\edg{\int_0^T |q_u|^2|Q_u|^2 du}}^{\frac{1}{2}}\|\varGamma^n\|_{\mathcal L^2}
=  \brak{E_Q\edg{\int_0^T |q_u|^2\bigg|\int_u^Tq_sds\bigg|^2 du}}^{\frac{1}{2}} \|\varGamma^n\|_{\mathcal L^2}\\
\le \brak{E_Q\edg{\brak{\int_0^T |q_u|^2du}^2}}^{\frac{1}{2}} \|\varGamma^n\|_{\mathcal L^2}
= \norm{q}_{\mathcal L^4}^2 \|\varGamma^n\|_{\mathcal L^2}\,,
\end{multline*}
where we used Fubini's theorem in the first equality above.
Since $\int qdW\in BMO$, the $\mathcal L^4$-norm of $q$ is finite\footnote{Recall that $BMO$
can be embedded into any $\mathcal H^p$-space , compare \citep[Section 2.1, p.~26]{Kazamaki01}.}
and the contradiction to \eqref{eq_same_order2}
is derived analogously to the argumentation 
following \eqref{eq_hoelder}.
Consequently, there exists a sequence $((\tilde\Delta^n,\tilde\varGamma^n))$ in the asymptotic convex hull of $((\Delta^n,\varGamma^n))$
and $(\Delta^Q,\varGamma^Q)\in\mathcal L^2\times\mathcal L^2$ such that 
$((\tilde\Delta^n,\tilde\varGamma^n))$ converges in $\mathcal L^2\times \mathcal L^2$ to $(\Delta^Q,\varGamma^Q)$.
On the side of $(Z^n)$ we pass to the corresponding sequence $(\tilde Z^n)$ and recall from the proof of
Lemma \ref{lemma21} that it is bounded in $\mathcal L^2$. 
Hence, there is a sequence in the asymptotic convex hull of $(\tilde Z^n)$, denoted likewise, 
that converges in $\mathcal L^2$ to some $Z^Q=z +\int \Delta^Qdu + \int \varGamma^QdW$. 
Of course, we pass the corresponding sequence on the side of $((\tilde\Delta^n,\tilde\varGamma^n))$ without violating the
convergence to $(\Delta^Q,\varGamma^Q)$.\\

\noindent\emph{Step 3: Lower Semicontinuity and convexity of $F$.}
In a next step we show that the earlier defined function $F(Z,\Delta,\varGamma) = E_Q[\int_0^T g_u(Z_u,\Delta_u,\varGamma_u) du - \int_0^T Z_udW_u]$
is lower semicontinuous and convex on $\mathcal L^2\times\mathcal L^2\times \mathcal L^2$ where $Z= z+\int \Delta du + \int \varGamma dW$.
Indeed, the part $E_Q[ \int_0^T g_u(Z_u,\Delta_u,\varGamma_u) du]$ is lower semicontinuous by \ref{pos}, \ref{lsc} and 
Fatou's lemma.
As to the second part, first 
observe that $\mathcal L^2$-convergence of 
$(\tilde Z^n)$ towards $Z^Q$ 
implies 
\begin{equation*}
 \abs{E_Q\edg{\int_0^T (\tilde Z^n_u- Z^Q_u)dW_u}} \underset{n\to\infty}{\longrightarrow} 0\,.
\end{equation*}
Furthermore, 
\ref{con} yields that $F$ is convex in $(Z, \Delta,\varGamma)$.\\

\noindent\emph{Step 4: Minimality of $(Z^Q,\Delta^Q,\varGamma^Q)$.}
We claim that $F(Z^Q,\Delta^Q,\varGamma^Q)=\inf_{(\Delta,\varGamma)\in\mathcal L^2\times\mathcal L^2}$ $F(Z,\Delta,\varGamma)$
which would then in turn finally imply \eqref{eq_attaining}.
To this end, it suffices to prove that $F(Z^Q,\Delta^Q,\varGamma^Q)\le\inf_{(\Delta,\varGamma)\in\mathcal L^2\times\mathcal L^2}F(Z,\Delta,\varGamma)$,
since the reverse inequality is naturally satisfied.
Observe now that
\begin{multline*}
\inf_{(\Delta,\varGamma)\in\mathcal L^2\times\mathcal L^2}F(Z,\Delta,\varGamma) = \lim_n F\brak{Z^n,\Delta^n,\varGamma^n} 
= \lim_n \sum_{k=n}^{M(n)}\lambda^{(n)}_k F\brak{Z^k,\Delta^k,\varGamma^k}\\
\ge \lim_n F\brak{\sum_{k=n}^{M(n)}\lambda^{(n)}_k Z^k,\sum_{k=n}^{M(n)}\lambda^{(n)}_k \Delta^k,\sum_{k=n}^{M(n)}\lambda^{(n)}_k\varGamma^k}\\
= \lim_n F\brak{\tilde Z^n,\tilde\Delta^n,\tilde\varGamma^n} \ge F(Z^Q,\Delta^Q,\varGamma^Q)
\end{multline*}
where we denoted by $\lambda^{(n)}_k$, $n\le k\le M(n)$, $\sum_k \lambda^{(n)}_k=1$ the convex weights of the sequence 
$((\tilde Z^n,\tilde\Delta^n,\tilde\varGamma^n))$ 
and made use of the convexity and lower semicontinuity of $F$.\\ 

\noindent\emph{Step 5: Uniqueness of $(Z^Q,\Delta^Q,\varGamma^Q)$.}
As to the uniqueness, assume that there are $(\Delta^1,\varGamma^1)$ and $(\Delta^2,\varGamma^2)$ with corresponding
$Z^1$ and $Z^2$, respectively, both attaining the supremum such that 
$P\otimes dt [(\Delta^1,\varGamma^1)\neq(\Delta^2,\varGamma^2)] >0$.
Setting $(\bar Z,\bar\Delta,\bar\varGamma):=\frac{1}{2}[(Z^1,\Delta^1,\varGamma^1)+(Z^2,\Delta^2,\varGamma^2)]$ together 
with $Q\sim P$ and the strict convexity of $F$
inherited by $g$ yields that 
$F(\bar Z,\bar\Delta,\bar\varGamma)<F(Z^1,\Delta^1,\varGamma^1)$, a contradiction to the optimality of $(Z^1,\Delta^1,\varGamma^1)$.  
\end{proof}
\begin{remark}
Since for a given $Q\in\mathcal Q$ with $\frac{dQ}{dP}\in L^\infty_b$ and a strictly convex generator the maximizer $(Z^Q,\Delta^Q,\varGamma^Q)$ 
is unique by the preceding proposition, it has to be (conditionally) optimal at all times $t\in[0,T]$.
Indeed, assume to the contrary the existence of $(\Delta,\varGamma)$ such that $z+\int_0^t\Delta_udu + \int_0^t\varGamma_udW_u = Z_t=Z^Q_t$ and
$E_Q[\int_t^T -g_u(Z_u,\Delta_u,\varGamma_u) du  + \int_t^T Z_u dW_u | \mathcal F_t] 
> E_Q[ \int_t^T- g_u(Z^Q_u,\Delta^Q_u,\varGamma^Q_u)du + \int_t^TZ^Q_u dW_u |\mathcal F_t]$ 
holds true for some $t\in[0,T]$.
Then, however, the concatenated processes $(\bar Z,\bar\Delta,\bar\varGamma):=(Z^Q,\Delta^Q,\varGamma^Q) 1_{[0,t]} + (Z,\Delta,\varGamma) 1_{]t,T]}$ 
satisfy 
$E_Q[\int_0^T -g_u(\bar Z_u,\bar\Delta_u,\bar\varGamma_u) du  + \int_0^T \bar Z_u dW_u ] 
> E_Q[ \int_0^T- g_u(Z^Q_u,\Delta^Q_u,\varGamma^Q_u)du + \int_0^T Z^Q_u dW_u ]$, which is a contradiction to the opimality of $(Z^Q,\Delta^Q,\varGamma^Q)$ 
at time zero. 
\end{remark}
Notice that, for $d=1$ and the case of a quadratic generator which is in addition independent of $z$, that is 
$g_u(\delta,\gamma)=|\delta|^2+|\gamma|^2$, 
the processes $(\Delta^Q,\varGamma^Q)$ attaining $\mathcal E^*_0(Q)$ can be explicitly computed and $(\Delta^Q_t,\varGamma^Q_t)$ depends 
on the whole path of $q$ up to time $t$, as illustrated in the following proposition.
It thus constitutes a useful tool for the characterization of the dual optimizers 
and its proof is closely related to the Euler-Lagrange equation arising in classical calculus of variation. 
\begin{proposition}\label{lem_app_EL}
Assume that $d=1$ and that $g$ is defined by $g_u(\delta,\gamma)=|\delta|^2+|\gamma|^2$.
For $Q\in\mathcal Q$ with $\frac{dQ}{dP}\in L^\infty_b$, let $(\Delta^Q,\varGamma^Q)$ 
be the optimizer attaining $\mathcal E^*_0(Q)$.
Then there exist $c_1,c_2\in\R$ such that
\begin{align}
& \Delta^Q_t = -\frac{1}{2}\int_0^t q_s ds + c_1 \label{eq_opt_delta}\\
& \varGamma^Q_t =  -\frac{1}{2}q_t\brak{\int_0^t q_s ds+c_2 }\label{eq_opt_gamma}\,,
\end{align}
for all $t\in[0,T]$. 
\end{proposition}
\begin{proof}
For the purpose of this proof we assume without loss of generality that $z=0$, since the
initial value does not affect the optimization with respect to $(\Delta,\varGamma)$.
Hence, the generator only depending on $\delta$ and $\gamma$ in combination with \eqref{eq_duality_05} 
gives $\mathcal E^*_0(Q) = -\inf_{(\Delta,\varGamma)} \{F_1(\Delta) + F_2(\varGamma)\}$
where $F_1(\Delta):=E_Q[\int_0^T(|\Delta_u|^2 - q_u\int_0^u\Delta_sds)du]$ 
and $F_2(\varGamma):=E_Q[\int_0^T(|\varGamma_u|^2 - q_u\int_0^uq_s\varGamma_sds)du]$.
We will proceed along an $\omega$-wise criterion of optimality, since any pair $(\Delta^Q,\varGamma^Q)$ that is optimal for almost all
$\omega\in\Omega$ then naturally also optimizes the expectation under $Q$.
The uniqueness obtained in Proposition \ref{lem_dual01} then assures that the path-wise optimizer
is the only one.
We define
\begin{equation*}
J_1(\Delta) = \int_0^T \brak{ |\Delta_u|^2 - q_u \int_0^u\Delta_s ds} du \enspace\quad\mbox{and}\enspace\quad
J_2(\varGamma) = \int_0^T \brak{|\varGamma_u|^2 - q_u \int_0^u q_s\varGamma_s ds} du 
\end{equation*}  
and observe that it is sufficient to elaborate how to obtain conditions for a minimizer of $J_1$, as the functional $J_2$ is of a
similar structure.
Introducing $X(u):=\int_0^u\Delta_s ds$ we obtain $X'(u):=\frac{d}{du}X(u) = \Delta_u$ and
\begin{equation*}
J_1(\Delta) = \tilde J_1(X) = \int_0^T L(u,X(u),X'(u)) du 
\end{equation*}
where $L(u,a,b) = |b|^2 - q_ua$.
If $X$ is a local minimum of $\tilde J_1$, then $\tilde J_1(X)\le \tilde J_1(X+\varepsilon \eta)$ for sufficiently small $\varepsilon>0$ and
all differentiable $\eta\in C([0,T],\R)$ the derivatives of which are square integrable and which satisfy  $\eta(0)=\eta(T)=0$.
In particular, with $\phi(\varepsilon):= \tilde J_1(X+\varepsilon\eta)$, 
it has to hold that $\frac{d}{d\varepsilon}\phi(\varepsilon)_{|_{\varepsilon=0}}=0$.
Using the specific form of $L$ we get
\begin{align*}
\frac{d}{d\varepsilon}\phi(\varepsilon)_{|_{\varepsilon=0}} &= \lim_{h\to 0} \frac{1}{h}\int_0^T
\edg{ L\big(u,X(u)+h\eta(u),X'(u)+h\eta'(u)\big)- L\big(u,X(u),X'(u)\big) }du\\
&= \lim_{h\to 0} \int_0^T \edg{ -q_u\eta(u)+2X'(u)\eta'(u) +h(\eta'(u))^2 }du\,.
\end{align*}
Having assumed $\eta'$ to be square integrable allows us to exchange limit and integration, yielding
\begin{equation}\label{eq_intbyparts}
0 = \frac{d}{d\varepsilon}\phi(\varepsilon)_{|_{\varepsilon=0}} =  \int_0^T \edg{ -q_u\eta(u)+2\Delta_u\eta'(u)}du\,.  
\end{equation}
Using integration by parts we obtain
\begin{equation*}
-\int_0^T q_u \eta(u)du = 
\left.\brak{\int_0^{u} -q_sds}\eta(u) \right|^T_0 \enspace-\enspace
\int_0^T \brak{\int_0^{u} -q_sds} \eta'(u) du\,. 
\end{equation*}
The first term on the right-hand side above vanishes and so, by plugging this back into \eqref{eq_intbyparts}  
we end up with
\begin{equation}\label{eq_EL02}
\int_0^T \brak{2\Delta_u + \int_0^u q_s ds }\eta'(u)du = 0 \,.
\end{equation}
Let us next introduce the constant $c:= \frac{1}{T} \int_0^T (2\Delta_u + \int_0^{u} q_sds)du$, of course depending on $\omega$,
and observe that, using $\int_0^T\eta'(u)du=0$, Equation \eqref{eq_EL02} may be rewritten as
\begin{equation}\label{eq_EL03}
\int_0^T \brak{2\Delta_u + \int_0^u q_s ds  -c}\eta'(u)du = 0\,. 
\end{equation}
Moreover, the function
\begin{equation*}
\bar\eta(t) := \int_0^t \brak{2\Delta_u + \int_0^{u} q_sds -c} du
\end{equation*}
satisfies $\bar\eta(0)=\bar\eta(T)=0$ by construction as well as $\bar\eta'(u) = 2\Delta_u + \int_0^{u} q_sds -c$,
which is square integrable for almost all $\omega\in\Omega$, since $\Delta$ and $q$ are square integrable\footnote{More precisely, it holds
$P(\int_0^T|q_u|^2du<\infty)=P(\int_0^T|\Delta_u|^2du<\infty)=1$.}.
Hence, \eqref{eq_EL03} applied to our particular function $\bar\eta$ yields
\begin{equation*}
\int_0^T \brak{2\Delta_u + \int_0^u q_s ds - c}^2du = 0
\end{equation*}
and we deduce that 
\begin{equation}\label{eq_EL04}
2\Delta_u + \int_0^u q_s ds  = \frac{1}{T}  \int_0^T \brak{2\Delta_r + \int_0^r q_s ds }dr \qquad\mbox{for almost all $u\in[0,T]$}\,.
\end{equation}
Specifically, \eqref{eq_EL04} shows that, for almost all $\omega\in\Omega$, 
there exists a set $\mathcal I(\omega)\subseteq [0,T]$ with Lebesgue measure $T$
such that, for all $u\in \mathcal I(\omega)$, 
\begin{equation}\label{eq_EL_Pdt}
2\Delta_u + \int_0^u q_s ds   = M \,,
\end{equation}
where $M$ of course depends on $\omega\in\Omega$ and is thus a random variable.
This in turn implies, for $dt$-almost all $u\in[0,T]$, the existence of $\Omega_u\subseteq\Omega$ with $P(\Omega_u)=1$
such that \eqref{eq_EL_Pdt} holds for all $\omega\in\Omega_u$.
In particular, on $\Omega_u$ the above $M$ equals an $\mathcal F_u$-measurable random variable.
We choose a sequence $(u_n)\subset [0,T]$ with $\lim_n u_n = 0$ and hence obtain that
on $\bar\Omega:= \bigcap_n\Omega_{u_n}$, where $P(\bar\Omega)=Q(\bar\Omega)=1$, $M$ equals an $\bigcap_n\mathcal F_{u_n}$-measurable random
variable and is thus $\mathcal F_0$-measurable, that means it is a constant on $\bar\Omega$, by the right-continuity of our filtration.
Since modifying our optimizer on a $Q$-nullset does not alter the value of the functional to be optimized, we have shown \eqref{eq_opt_delta}
by putting $c_1:=\frac{M}{2}$.

As to the case of our optimal $\varGamma^Q$, assume first that $q_u\neq 0$ for all $u\in[0,T]$ and observe that, with $Y(u):=\int_0^u q_s\varGamma_s ds$, 
we obtain $J_2(\varGamma) = \tilde J_2(Y) = \int_0^T K(u,Y(u),Y'(u)) du $
where we set $K(u,a,b) = (1/q_u)^2 b^2 - q_ua$.
Thus, an argumentation identical to that above would yield \eqref{eq_opt_gamma} in that case.
Furthermore, it holds that $\lim_{q_u\to 0}\varGamma^Q_u=0$, a value that is consistent 
with the ``pointwise''  minimization consideration that $\argmin_{\varGamma_u}\{|\varGamma_u|^2 - q_u\int_{[0,u)}q_s\varGamma_sds\}|_{q_u=0}=0$
and hence justifies expression \eqref{eq_opt_gamma}. 
\end{proof}
For $\xi\in L^1$ such that $\mathcal A(\xi,g,z)\neq\emptyset$, the following theorem states that, given the existence of an equivalent 
probability measure $\hat Q\in\mathcal Q$ such that
the $\sup$ in \eqref{eq_duality_03b} is attained, the BSDE with generator $g$ and terminal condition $\xi$ admits a solution
under constraints.
Conditions guaranteeing that the rather technical assumptions of the ensuing theorem are satisfied
are subject to further research.
\begin{theorem}
Assume that, for $\xi\in L^1$ with  $\mathcal A(\xi,g,z)\neq\emptyset$, there exists a $\hat Q\in\mathcal Q$ 
with $\frac{dQ}{dP}\in L^\infty_b$ such that 
$\mathcal E^g_0(\xi) = E_{\hat Q}[\xi] - \mathcal E_0^*(\hat Q)$. 
Then there exists a solution $(Y,Z)\in\mathcal A(\xi,g,z)$ of the BSDE with parameters $(\xi,g)$.
\end{theorem}
\begin{proof}
Starting with \eqref{eq_duality_03b} in combination with Proposition \ref{lem_dual01}, it holds
\begin{equation}\label{eq_duality_04}
\mathcal E^g_0(\xi) = E_{\hat Q}[\xi] - \mathcal E_0^*(\hat Q)
= E_{\hat Q}\left[\xi+ \int_0^T g_u(Z^{\hat Q}_u,\Delta^{\hat Q}_u,\varGamma^{\hat Q}_u) du  - \int_0^T Z^{\hat Q}_udW_u \right]\,.
\end{equation}
We recall that $\mathcal A(\xi,g,z)\neq\emptyset$ and $\xi\in L^1$.
Hence, by Theorem \ref{prop1} there exists $(\Delta,\varGamma,Z)$ such that
\begin{equation*}
E_{\hat Q}\left[\xi+ \int_0^T g_u(Z^{\hat Q}_u,\Delta^{\hat Q}_u,\varGamma^{\hat Q}_u) du  - \int_0^T Z^{\hat Q}_udW_u \right] 
- \int_0^T g_u(Z_u,\Delta_u,\varGamma_u) du  + \int_0^T Z_udW_u \ge \xi
\end{equation*}
holds true.
Taking expectation under $\hat Q$ on both sides of the inequality above yields
\begin{multline*}
E_{\hat Q}\left[- \int_0^T g_u(Z^{\hat Q}_u,\Delta^{\hat Q}_u,\varGamma^{\hat Q}_u) du  + \int_0^T Z^{\hat Q}_udW_u \right]
\\ \le E_{\hat Q}\left[- \int_0^T g_u(Z_u,\Delta_u,\varGamma_u) du  + \int_0^T Z_udW_u \right]\,.
\end{multline*}
However, the expression on the left-hand side is maximal for $(Z^{\hat Q},\Delta^{\hat Q},\varGamma^{\hat Q})$ 
by means of Proposition \ref{lem_dual01} 
and thus equality has to hold.
Hence, it follows that $\mathcal E^*_0 (\hat Q) = E_{\hat Q}[ -\int_0^T g_u(Z_u,\Delta_u,$ $\varGamma_u) du  + \int_0^T Z_udW_u]$.
By plugging this back into \eqref{eq_duality_04} we obtain
\begin{equation*}
\mathcal E^g_0(\xi) = E_{\hat Q}[\xi]+ E_{\hat Q}[ \int_0^T g_u(Z_u,\Delta_u,\varGamma_u) du  - \int_0^T Z_udW_u]
\end{equation*}
which is equivalent to 
\begin{equation*}
E_{\hat Q}\left[ \mathcal E^g_0(\xi) - \int_0^T g_u(Z_u,\Delta_u,\varGamma_u) du  + \int_0^T Z_udW_u - \xi \right] = 0\,.
\end{equation*}
Since the expression within the expectation is $P$- and thereby also $\hat Q$-almost surely positive, we finally conclude
that
\begin{equation*}
\mathcal E^g_0(\xi) - \int_0^T g_u(Z_u,\Delta_u,\varGamma_u) du  + \int_0^T Z_udW_u = \xi\,,
\end{equation*}
and thus $(\mathcal E^g_0(\xi) - \int_0^\cdot g(Z,\Delta,\varGamma) du  + \int_0^\cdot ZdW,Z)$ 
constitutes a solution of the BSDE with parameters $(\xi,g)$.
\end{proof}

%
\bibliographystyle{abbrvnat}
\bibliography{bibliography}
\end{document}